\documentclass[a4paper,oneside,10pt]{article}%
\usepackage{amsmath}
\usepackage{amsfonts}
\usepackage{amssymb}
\usepackage{graphicx}
\usepackage[square,numbers,sort&compress]{natbib}%
\providecommand{\U}[1]{\protect\rule{.1in}{.1in}}

\newtheorem{theorem}{Theorem}[section]

\newtheorem{corollary}[theorem]{Corollary}

\newtheorem{example}[theorem]{Example}

\newtheorem{lemma}[theorem]{Lemma}

\newtheorem{remark}[theorem]{Remark}

\newenvironment{proof}[1][Proof]{\noindent\textbf{#1.} }{\  $\Box$}
\numberwithin{equation}{section}

\begin{document}

\title{Direct Method on Stochastic Maximum Principle for Optimization with Recursive Utilities}
\author{Mingshang Hu \thanks{Qilu Institute of Finance, Shandong University, Jinan,
Shandong 250100, PR China. humingshang@sdu.edu.cn. Research supported by NSF
(No. 11201262 and 10921101) and Shandong Province (No.BS2013SF020 and
ZR2014AP005) and the 111 Project (No. B12023) }}
\date{}
\maketitle

\textbf{Abstract}. We obtain the variational equations for backward stochastic
differential equations in recursive stochastic optimal control problems, and
then get the maximum principle which is novel. The control domain need not be
convex, and the generator of the backward stochastic differential equation can
contain $z$.



\addcontentsline{toc}{section}{\hspace*{1.8em}Abstract}

\section{Introduction}

Let $(\Omega,\mathcal{F},P)$ be a probability space and let $W$ be a
$d$-dimensional Brownian motion. The filtration $\{\mathcal{F}_{t}:t\geq0\}$
is generated by $W$, i.e.,%
\[
\mathcal{F}_{t}:=\sigma\{W(s):s\leq t\}\vee\mathcal{N},
\]
where $\mathcal{N}$ is all $P$-null sets. Let $U$ be a set in $\mathbb{R}^{k}$
and $T>0$ be a given terminal time. Set%
\[
\mathcal{U}[0,T]:=\{(u(s))_{s\in\lbrack0,T]}:u\text{ is progressively
measurable, }u(s)\in U\text{ and }E[\int_{0}^{T}|u(s)|^{\beta}ds]<\infty\text{
for all }\beta>0\},
\]
where $U$ is called the control domain and $\mathcal{U}[0,T]$ is called the
set of all admissible controls. In fact, we just need $E[\int_{0}%
^{T}|u(s)|^{\beta_{0}}ds]<\infty$ for some $\beta_{0}>0$. For simplicity, we
do not explicitly give this $\beta_{0}$ in this paper. Consider the following
state equation:%
\begin{equation}
\left\{
\begin{array}
[c]{l}%
dx(t)=b(t,x(t),u(t))dt+\sigma(t,x(t),u(t))dW(t),\\
x(0)=x_{0}\in\mathbb{R}^{n},
\end{array}
\right.  \label{state1}%
\end{equation}
where $b:[0,T]\times\mathbb{R}^{n}\times\mathbb{R}^{k}\rightarrow
\mathbb{R}^{n}$, $\sigma:[0,T]\times\mathbb{R}^{n}\times\mathbb{R}%
^{k}\rightarrow\mathbb{R}^{n\times d}$. The cost functional is defined by%
\begin{equation}
J(u(\cdot))=E[\phi(x(T)+\int_{0}^{T}f(t,x(t),u(t))dt], \label{cost1}%
\end{equation}
where $\phi:\mathbb{R}^{n}\rightarrow\mathbb{R}$, $g:[0,T]\times\mathbb{R}%
^{n}\times\mathbb{R}^{k}\rightarrow\mathbb{R}$. The classical stochastic
optimal control problem is to minimize $J(u(\cdot))$ over $\mathcal{U}[0,T]$.
If there exists a $\bar{u}\in\mathcal{U}[0,T]$ such that%
\[
J(\bar{u}(\cdot))=\inf_{u\in\mathcal{U}[0,T]}J(u(\cdot)),
\]
$\bar{u}$ is called an optimal control. $\bar{x}(\cdot)$, which is the
solution of state equation (\ref{state1}) corresponding to $\bar{u}$, is
called an optimal trajectory. The maximum principle is to find the necessary
condition for the optimal control $\bar{u}$.

The method for deriving the maximum principle is the variational principle.
When $U$ is not convex, we use the spike variation method. More precisely, let
$\varepsilon>0$ and $E_{\varepsilon}\subset\lbrack0,T]$ with $|E_{\varepsilon
}|=\varepsilon$, define%
\[
u^{\varepsilon}(t)=\bar{u}(t)I_{E_{\varepsilon}^{c}}(t)+uI_{E_{\varepsilon}%
}(t),
\]
where $u\in U$. This $u^{\varepsilon}$ is called a spike variation of the
optimal control $\bar{u}$. For deriving the maximum principle, we only need to
use $E_{\varepsilon}=[s,s+\varepsilon]$ for $s\in\lbrack0,T-\varepsilon]$ and
$\varepsilon>0$. The difficulty of the classical stochastic optimal control
problem is the variational equation for $x(\cdot)$, which is completely
different from the deterministic optimal control problem. Peng \cite{P90}
first considered the second-order term in the Taylor expansion of the
variation and obtained the maximum principle for the classical stochastic
optimal control problem.

Consider the following backward stochastic differential equation (BSDE for
short):%
\begin{equation}
y(t)=\phi(x(T))+\int_{t}^{T}f(s,x(s),y(s),z(s),u(s))ds-\int_{t}^{T}%
z(s)dW(s),\label{state2}%
\end{equation}
where $\phi:\mathbb{R}^{n}\rightarrow\mathbb{R}$, $f:[0,T]\times\mathbb{R}%
^{n}\times\mathbb{R}\times\mathbb{R}^{d}\times\mathbb{R}^{k}\rightarrow
\mathbb{R}$. Pardoux and Peng \cite{PP90} first obtained that the BSDE
(\ref{state2}) has a unique solution $(y(\cdot),z(\cdot))$. Duffie and Epstein
\cite{DE} introduced the notion of recursive utilities in continuous time,
which is a kind of BSDE with $f$ independent of $z$. In \cite{EPQ, EPQ1}, the
authors extended the recursive utility to the case where $f$ contains $z$. The
term $z$ can be interpreted as an ambiguity aversion term in the market (see
\cite{CE}).

When $f$ is independent of $(y,z)$, it is easy to check that $y(0)=E[\phi
(x(T))+\int_{0}^{T}f(t,x(t),u(t))dt]$. So it is natural to extend the
classical stochastic optimal control problem to the recursive case. Consider
the control system which contains equations (\ref{state1}) and (\ref{state2}).
Define the cost functional%
\begin{equation}
J(u(\cdot))=y(0).\label{cost2}%
\end{equation}
The recursive stochastic optimal control problem is to minimize $J(u(\cdot))$
in (\ref{cost2}) over $\mathcal{U}[0,T]$. When the control domain $U$ is
convex, the local maximum principle for this problem can be found in
\cite{P93, DZ, JZ, SW, W1, X} and the references therein. In this paper, the
control domain $U$ is not necessarily convex, we must obtain the global
maximum principle by the spike variation method.

One direct method for treating this problem is to consider the second-order
terms in the Taylor expansion of the variation for the BSDE (\ref{state2}) as
in \cite{P90}. When $f$ depends nonlinearly on $z$, there are two major
difficulties (see \cite{Y1}): (i) What is the second-order variational
equation for the BSDE (\ref{state2}), which is not the one similar in
\cite{P90}. (ii) How to get the second-order adjoint equation which seems to
be unexpectedly complicated due to the quadratic form with respect to the
variation of $z$.

Based on these difficulties, Peng \cite{P98} proposed the following open
problem in page 269:

\textquotedblleft The corresponding `global maximum principle' for the case
where $f$ depends nonlinearly on $z$ is open, except for some special case."

Recently, a new method for treating this problem is to see $z(\cdot)$ as a
control and the terminal condition $y(T)=\phi(x(T))$ as a constraint, then use
the Ekeland variational principle to obtain the maximum principle. This idea
was used in \cite{KZ, LZ} for studying the backward linear-quadratic optimal
control problem, and then was used in \cite{W, Y1} for studying the recursive
stochastic optimal control problem. But the maximum principle contains unknown parameters.

In this paper, we overcome the two major difficulties in the above direct
method. The second-order variational equation for the BSDE (\ref{state2}) and
the maximum principle have been obtained. The main difference of the
variational equations with the ones in \cite{P90} lies in the term $\langle
p(t),\delta\sigma(t)\rangle I_{E_{\varepsilon}}(t)$ (see equation
(\ref{adjoint1}) in Section 3 for the definition of $p(t)$) in the variation
of $z$, which is $O(\varepsilon)$ for any order expansion of $f$. So it is not
helpful to use the second-order Taylor expansion for treating this term.
Moreover, we also obtain the structure of the variation for $(y,z)$ and the
variation for $x$. Based on this, we can get the second-order adjoint
equation. Due to the term $\langle p(t),\delta\sigma(t)\rangle
I_{E_{\varepsilon}}(t)$ in the variation of $z$, our global maximum principle
is novel and different from the one in \cite{W, Y1}, which completely solves
Peng's open problem. Furthermore, our maximum principle is stronger than the
one in \cite{W, Y1}  (see Example \ref{newexa1}).

This paper is organized as follows. In Section 2, we give some basic results
and the idea for the variation of BSDE. The variational equations for BSDE and
the maximum principle have been obtained in Section 3. In Section 4, we obtain
the maximum principle for the control system with state constraint.

\section{Preliminaries and idea for variation of BSDE}

The results of this part can be found in \cite{P90, YZ}. For the simplicity of
presentation, we suppose $d=1$. We need the following assumption:

\begin{description}
\item[(A1)] $b$, $\sigma$ are twice continuously differentiable with respect
to $x$; $b$, $b_{x}$, $b_{xx}$, $\sigma$, $\sigma_{x}$, $\sigma_{xx}$ are
continous in $(x,u)$; $b_{x}$, $b_{xx}$, $\sigma_{x}$, $\sigma_{xx}$ are
bounded; $b$, $\sigma$ are bounded by $C(1+|x|+|u|)$.
\end{description}

Let $\bar{u}(\cdot)$ be the optimal control for the cost function defined in
(\ref{cost1}) and let $\bar{x}(\cdot)$ be the corresponding solution of
equation (\ref{state1}). Similarly, we define $(x^{\varepsilon}(\cdot
),u^{\varepsilon}(\cdot))$. Set
\begin{equation}%
\begin{array}
[c]{ll}%
b(\cdot)=(b^{1}(\cdot),\ldots,b^{n}(\cdot))^{T}, & \sigma(\cdot)=(\sigma
^{1}(\cdot),\ldots,\sigma^{n}(\cdot))^{T},\\
b(t)=b(t,\bar{x}(t),\bar{u}(t)), & \delta b(t)=b(t,\bar{x}(t),u)-b(t),
\end{array}
\label{nota1}%
\end{equation}
similar for $b_{x}(t)$, $b_{xx}^{i}(t)$, $\delta b_{x}(t)$, $\delta b_{xx}%
^{i}(t)$, $\sigma(t)$, $\sigma_{x}(t)$, $\sigma_{xx}^{i}(t)$, $\delta
\sigma(t)$, $\delta\sigma_{x}(t)$ and $\delta\sigma_{xx}^{i}(t)$, $i\leq n$,
where $b_{x}=(b_{x_{j}}^{i})_{i,j}$. Let $x_{i}(\cdot)$, $i=1$, $2$, be the
solution of the following stochastic differential equations (SDEs for short):%

\begin{equation}
\left\{
\begin{array}
[c]{l}%
dx_{1}(t)=b_{x}(t)x_{1}(t)dt+\{\sigma_{x}(t)x_{1}(t)+\delta\sigma
(t)I_{E_{\varepsilon}}(t)\}dW(t),\\
x_{1}(0)=0,
\end{array}
\right.  \label{vari1}%
\end{equation}

\bigskip%
\begin{equation}
\left\{
\begin{array}
[c]{rl}%
dx_{2}(t)= & \{b_{x}(t)x_{2}(t)+\delta b(t)I_{E_{\varepsilon}}(t)+\frac{1}%
{2}b_{xx}(t)x_{1}(t)x_{1}(t)\}dt\\
& +\{\sigma_{x}(t)x_{2}(t)+\delta\sigma_{x}(t)x_{1}(t)I_{E_{\varepsilon}%
}(t)+\frac{1}{2}\sigma_{xx}(t)x_{1}(t)x_{1}(t)\}dW(t),\\
x_{2}(0)= & 0,
\end{array}
\right.  \label{var2}%
\end{equation}
where $b_{xx}(t)x_{1}(t)x_{1}(t)=(\mathrm{tr}[b_{xx}^{1}(t)x_{1}%
(t)x_{1}(t)^{T}],\ldots,\mathrm{tr}[b_{xx}^{n}(t)x_{1}(t)x_{1}(t)^{T}])^{T}$
and similar for $\sigma_{xx}(t)x_{1}(t)x_{1}(t)$.

\begin{theorem}
\label{th1} Suppose (A1) holds. Then, for any $\beta\geq1$,%
\begin{equation}
E[\sup_{t\in\lbrack0,T]}|x^{\varepsilon}(t)-\bar{x}(t)|^{2\beta}%
]=O(\varepsilon^{\beta}), \label{est1}%
\end{equation}%
\begin{equation}
E[\sup_{t\in\lbrack0,T]}|x_{1}(t)|^{2\beta}]=O(\varepsilon^{\beta}),
\label{est2}%
\end{equation}%
\begin{equation}
E[\sup_{t\in\lbrack0,T]}|x_{2}(t)|^{2\beta}]=O(\varepsilon^{2\beta}),
\label{est3}%
\end{equation}%
\begin{equation}
E[\sup_{t\in\lbrack0,T]}|x^{\varepsilon}(t)-\bar{x}(t)-x_{1}(t)|^{2\beta
}]=O(\varepsilon^{2\beta}), \label{est4}%
\end{equation}%
\begin{equation}
E[\sup_{t\in\lbrack0,T]}|x^{\varepsilon}(t)-\bar{x}(t)-x_{1}(t)-x_{2}%
(t)|^{2\beta}]=o(\varepsilon^{2\beta}). \label{est5}%
\end{equation}
Moreover, we have the following expansion:%
\begin{equation}%
\begin{array}
[c]{rl}%
E[\phi(x^{\varepsilon}(T))]-E[\phi(\bar{x}(T))]= & E[\langle\phi_{x}(\bar
{x}(T)),x_{1}(T)+x_{2}(T)\rangle]\\
& +E[\frac{1}{2}\langle\phi_{xx}(\bar{x}(T))x_{1}(T),x_{1}(T)\rangle
]+o(\varepsilon).
\end{array}
\label{est6}%
\end{equation}

\end{theorem}

From this result, we can simply write $x_{1}(\cdot)=O(\sqrt{\varepsilon})$,
$x_{2}(\cdot)=O(\varepsilon)$ and%
\begin{equation}
x^{\varepsilon}(\cdot)=\bar{x}(\cdot)+x_{1}(\cdot)+x_{2}(\cdot)+o(\varepsilon
).\label{var3}%
\end{equation}
This equation is the variational equation for SDE (\ref{state1}) obtained by
Peng in \cite{P90}. Furthermore, he proved that%
\[
E[\langle\phi_{x}(\bar{x}(T)),x_{1}(T)\rangle]=O(\varepsilon).
\]
This is supprise, because $x_{1}(T)=O(\sqrt{\varepsilon})$. The reason is that
the order$\sqrt{\varepsilon}$ term is in the stochastic integral with respect
to $W$. Notice that $z$ in BSDE (\ref{state2}) is still in the stochastic
integral with respect to $W$. If we combine the order $\sqrt{\varepsilon}$
term and $z$, the other terms are $O(\varepsilon)$. Maybe we can get the
variational equation. In Section 3, we will give a rigorous proof of this idea.

\section{Variational equation for BSDE and maximum principle}

\subsection{Peng's open problem}

Suppose $d=1$ for the simplicity of presentation. The results for $d>1$ will
be given in the next subsection.

We consider the control system: SDE (\ref{state1}) and BSDE (\ref{state2}).
The cost function $J(u(\cdot))$ is defined in (\ref{cost2}). The control
problem is to minimize $J(u(\cdot))$ over $\mathcal{U}[0,T]$.

We need the following assumption:

\begin{description}
\item[(A2)] $f$, $\phi$ are twice continuously differentiable with respect to
$(x,y,z)$; $f$, $Df$, $D^{2}f$ are continous in $(x,y,z,u)$; $Df$, $D^{2}f$,
$\phi_{xx}$ are bounded; $\phi_{x}$ are bounded by $C(1+|x|)$.
\end{description}

$Df$ is the gradient of $f$ with respect to $(x,y,z)$, $D^{2}f$ is the Hessian
matrix of $f$ with respect to $(x,y,z)$.

Let $\bar{u}(\cdot)$ be the optimal control and let $(\bar{x}(\cdot),\bar
{y}(\cdot),\bar{z}(\cdot))$ be the corresponding solution of equations
(\ref{state1}) and (\ref{state2}). Similarly, we define $(x^{\varepsilon
}(\cdot),y^{\varepsilon}(\cdot),z^{\varepsilon}(\cdot),u^{\varepsilon}%
(\cdot))$.

In order to obtain the variational equation for BSDE (\ref{state2}), we
introduce the following adjoint equations:%
\begin{equation}
\left\{
\begin{array}
[c]{l}%
-dp(t)=\{f_{y}(t)p(t)+[f_{z}(t)\sigma_{x}^{T}(t)+b_{x}^{T}(t)]p(t)+f_{z}%
(t)q(t)+\sigma_{x}^{T}(t)q(t)+f_{x}(t)\}dt-q(t)dW(t),\\
p(T)=\phi_{x}(\bar{x}(T)),
\end{array}
\right.  \label{adjoint1}%
\end{equation}%
\begin{equation}
\left\{
\begin{array}
[c]{rl}%
-dP(t)= & \{f_{y}(t)P(t)+[f_{z}(t)\sigma_{x}^{T}(t)+b_{x}^{T}%
(t)]P(t)+P(t)[f_{z}(t)\sigma_{x}(t)+b_{x}(t)]+\sigma_{x}^{T}(t)P(t)\sigma
_{x}(t)\\
& +f_{z}(t)Q(t)+\sigma_{x}^{T}(t)Q(t)+Q(t)\sigma_{x}(t)+b_{xx}^{T}%
(t)p(t)+\sigma_{xx}^{T}(t)[f_{z}(t)p(t)+q(t)]\\
& +[I_{n\times n},p(t),\sigma_{x}^{T}(t)p(t)+q(t)]D^{2}f(t)[I_{n\times
n},p(t),\sigma_{x}^{T}(t)p(t)+q(t)]^{T}\}dt-Q(t)dW(t),\\
P(T)= & \phi_{xx}(\bar{x}(T)),
\end{array}
\right.  \label{adjoint2}%
\end{equation}
where%
\begin{equation}
f_{x}(t)=f_{x}(t,\bar{x}(t),\bar{y}(t),\bar{z}(t),\bar{u}(t)),\text{
}p(t)=(p^{1}(t),\ldots,p^{n}(t))^{T},\text{ }b_{xx}^{T}(t)p(t)=\sum_{i=1}%
^{n}p^{i}(t)b_{xx}^{i}(t), \label{nota2}%
\end{equation}
similar for $f_{y}(t)$, $f_{z}(t)$, $D^{2}f(t)$ and $\sigma_{xx}^{T}%
(t)[f_{z}(t)p(t)+q(t)]$. By the assumptions (A1) and (A2), it is easy to check
that the above BSDEs have unique solutions $(p(\cdot),q(\cdot))$ and
$(P(\cdot),Q(\cdot))$ respectively.

Applying It\^{o}'s formula to $\langle p(t),x_{1}(t)+x_{2}(t)\rangle+\frac
{1}{2}\langle P(t)x_{1}(t),x_{1}(t)\rangle$, we can get the following lemma by
simple calculation.

\begin{lemma}
\label{pr2}We have%
\begin{equation}%
\begin{array}
[c]{l}%
\langle p(T),x_{1}(T)+x_{2}(T)\rangle+\frac{1}{2}\langle P(T)x_{1}%
(T),x_{1}(T)\rangle=\langle p(t),x_{1}(t)+x_{2}(t)\rangle+\frac{1}{2}\langle
P(t)x_{1}(t),x_{1}(t)\rangle\\
+\int_{t}^{T}\{[\langle p(s),\delta b(s)\rangle+\langle q(s),\delta
\sigma(s)\rangle+\frac{1}{2}\langle P(s)\delta\sigma(s),\delta\sigma
(s)\rangle]I_{E_{\varepsilon}}(s)\\
-\langle f_{x}(s)+f_{y}(s)p(s)+f_{z}(s)[\sigma_{x}^{T}(s)p(s)+q(s)],x_{1}%
(s)+x_{2}(s)\rangle\\
-\frac{1}{2}\langle\{f_{y}(s)P(s)+f_{z}(s)[\sigma_{xx}^{T}(s)p(s)+P(s)\sigma
_{x}(s)+\sigma_{x}^{T}(s)P(s)+Q(s)]\\
+[I_{n\times n},p(s),\sigma_{x}^{T}(s)p(s)+q(s)]D^{2}f(s)[I_{n\times
n},p(s),\sigma_{x}^{T}(s)p(s)+q(s)]^{T}\}x_{1}(s),x_{1}(s)\rangle\}ds\\
+\int_{t}^{T}\{\langle p(s),\delta\sigma(s)\rangle I_{E_{\varepsilon}%
}(s)+\langle\sigma_{x}^{T}(s)p(s)+q(s),x_{1}(s)+x_{2}(s)\rangle\\
+\langle\delta\sigma_{x}^{T}(s)p(s)+\frac{1}{2}P(s)\delta\sigma(s)+\frac{1}%
{2}P^{T}(s)\delta\sigma(s),x_{1}(s)\rangle I_{E_{\varepsilon}}(s)\\
+\frac{1}{2}\langle\lbrack\sigma_{xx}^{T}(s)p(s)+P(s)\sigma_{x}(s)+\sigma
_{x}^{T}(s)P(s)+Q(s)]x_{1}(s),x_{1}(s)\rangle\}dW(s)\\
+\int_{t}^{T}\langle\delta\sigma_{x}^{T}(s)q(s)+\frac{1}{2}[\sigma_{x}%
^{T}(s)P(s)+\sigma_{x}^{T}(s)P^{T}(s)+Q(s)+Q^{T}(s)]\delta\sigma
(s),x_{1}(s)\rangle I_{E_{\varepsilon}}(s)ds.
\end{array}
\label{vm1}%
\end{equation}

\end{lemma}

By Theorem \ref{th1}, it is easy to check that the last line of equation
(\ref{vm1}) is $o(\varepsilon)$ and%
\[
\phi(x^{\varepsilon}(T))=\phi(\bar{x}(T))+\langle p(T),x_{1}(T)+x_{2}%
(T)\rangle+\frac{1}{2}\langle P(T)x_{1}(T),x_{1}(T)\rangle+o(\varepsilon).
\]
We set%
\[%
\begin{array}
[c]{rl}%
\bar{y}^{\varepsilon}(t)= & y^{\varepsilon}(t)-[\langle p(t),x_{1}%
(t)+x_{2}(t)\rangle+\frac{1}{2}\langle P(t)x_{1}(t),x_{1}(t)\rangle],\\
\bar{z}^{\varepsilon}(t)= & z^{\varepsilon}(t)-\{\langle p(t),\delta
\sigma(t)\rangle I_{E_{\varepsilon}}(t)+\langle\sigma_{x}^{T}%
(t)p(t)+q(t),x_{1}(t)+x_{2}(t)\rangle\\
& +\langle\delta\sigma_{x}^{T}(t)p(t)+\frac{1}{2}P(t)\delta\sigma(t)+\frac
{1}{2}P^{T}(t)\delta\sigma(t),x_{1}(t)\rangle I_{E_{\varepsilon}}(t)\\
& +\frac{1}{2}\langle\lbrack\sigma_{xx}^{T}(t)p(t)+P(t)\sigma_{x}%
(t)+\sigma_{x}^{T}(t)P(t)+Q(t)]x_{1}(t),x_{1}(t)\rangle\}.
\end{array}
\]

\begin{remark}
\label{newre12} It is important to note that the term $\langle p(t),\delta
\sigma(t)\rangle I_{E_{\varepsilon}}(t)$ in the $\bar{z}^{\varepsilon}(t)$
comes from the It\^{o}'s formula for $\langle p(t),x_{1}(t)\rangle$. Note that
$(\langle p(t),\delta\sigma(t)\rangle I_{E_{\varepsilon}}(t))^{i}=(\langle
p(t),\delta\sigma(t)\rangle)^{i}I_{E_{\varepsilon}}(t)$ for any integer
$i\geq1$, so it is not helpful to use the second-order Taylor expansion for
treating this term, which is completely different from the terms $x_{1}(t)$,
$x_{2}(t)$, $x_{1}(t)I_{E_{\varepsilon}}(t)$ and $x_{1}(t)(x_{1}(t))^{T}$. In
the following, we will show that this term is indeed in the variation of $z$,
which is important for getting the variational equations for $(y,z)$.
\end{remark}

Then by Lemma \ref{pr2}, we can get%
\begin{equation}%
\begin{array}
[c]{cl}%
\bar{y}^{\varepsilon}(t)= & \phi(\bar{x}(T))+o(\varepsilon)+\int_{t}%
^{T}\{[\langle p(s),\delta b(s)\rangle+\langle q(s),\delta\sigma
(s)\rangle+\frac{1}{2}\langle P(s)\delta\sigma(s),\delta\sigma(s)\rangle
]I_{E_{\varepsilon}}(s)\\
& +f(s,x^{\varepsilon}(s),y^{\varepsilon}(s),z^{\varepsilon}(s),u^{\varepsilon
}(s))-\langle f_{x}(s)+f_{y}(s)p(s)+f_{z}(s)[\sigma_{x}^{T}(s)p(s)+q(s)],x_{1}%
(s)+x_{2}(s)\rangle\\
& -\frac{1}{2}\langle\{f_{y}(s)P(s)+f_{z}(s)[\sigma_{xx}^{T}(s)p(s)+P(s)\sigma
_{x}(s)+\sigma_{x}^{T}(s)P(s)+Q(s)]\\
& +[I_{n\times n},p(s),\sigma_{x}^{T}(s)p(s)+q(s)]D^{2}f(s)[I_{n\times
n},p(s),\sigma_{x}^{T}(s)p(s)+q(s)]^{T}\}x_{1}(s),x_{1}(s)\rangle\}ds\\
& -\int_{t}^{T}\bar{z}^{\varepsilon}(s)dW(s).
\end{array}
\label{cbsde1}%
\end{equation}

\begin{remark}
By the standard estimates of BSDEs, we can show that $\bar{y}^{\varepsilon
}(t)-\bar{y}(t)=O(\varepsilon)$, $\bar{z}^{\varepsilon}(t)-\bar{z}%
(t)=O(\varepsilon)$ (see the proof of Theorem \ref{th3}), which is the reason
for constructing adjoint equations (\ref{adjoint1}) and (\ref{adjoint2}).
\end{remark}

Suppose that%
\begin{equation}%
\begin{array}
[c]{l}%
\bar{y}^{\varepsilon}(t)=\bar{y}(t)+\hat{y}(t)+o(\varepsilon),\\
\bar{z}^{\varepsilon}(t)=\bar{z}(t)+\hat{z}(t)+o(\varepsilon).
\end{array}
\label{newassum}%
\end{equation}
Then from BSDE (\ref{cbsde1}) and the Taylor expansion, we consider the
following BSDE:%
\begin{equation}%
\begin{array}
[c]{cl}%
\hat{y}(t)= & \int_{t}^{T}\{f_{y}(s)\hat{y}(s)+f_{z}(s)\hat{z}(s)+[\langle
p(s),\delta b(s)\rangle+\langle q(s),\delta\sigma(s)\rangle+\frac{1}{2}\langle
P(s)\delta\sigma(s),\delta\sigma(s)\rangle\\
& +f(s,\bar{x}(s),\bar{y}(s),\bar{z}(s)+\langle p(s),\delta\sigma
(s)\rangle,u)-f(s,\bar{x}(s),\bar{y}(s),\bar{z}(s),\bar{u}%
(s))]I_{E_{\varepsilon}}(s)\}ds\\
& -\int_{t}^{T}\hat{z}(s)dW(s).
\end{array}
\label{vbsde1}%
\end{equation}
In the following theorem, we will prove the above assumption (\ref{newassum}).

\begin{theorem}
\label{th3}Suppose (A1) and (A2) hold. Then, for any $\beta\geq2$,%
\begin{equation}
E[\sup_{t\in\lbrack0,T]}|\bar{y}^{\varepsilon}(t)-\bar{y}(t)|^{2}+\int_{0}%
^{T}|\bar{z}^{\varepsilon}(t)-\bar{z}(t)|^{2}dt]=O(\varepsilon^{2}),
\label{vbs1}%
\end{equation}%
\begin{equation}
E[\sup_{t\in\lbrack0,T]}|\bar{y}^{\varepsilon}(t)-\bar{y}(t)|^{\beta}%
+(\int_{0}^{T}|\bar{z}^{\varepsilon}(t)-\bar{z}(t)|^{2}dt)^{\beta
/2}]=o(\varepsilon^{\beta/2}), \label{vbs2}%
\end{equation}%
\begin{equation}
E[\sup_{t\in\lbrack0,T]}|\hat{y}(t)|^{2}+\int_{0}^{T}|\hat{z}(t)|^{2}%
dt]=O(\varepsilon^{2}), \label{vbs3}%
\end{equation}%
\begin{equation}
E[\sup_{t\in\lbrack0,T]}|\bar{y}^{\varepsilon}(t)-\bar{y}(t)-\hat{y}%
(t)|^{2}+\int_{0}^{T}|\bar{z}^{\varepsilon}(t)-\bar{z}(t)-\hat{z}%
(t)|^{2}dt]=o(\varepsilon^{2}). \label{vbs4}%
\end{equation}

\end{theorem}

\begin{proof}
We first prove (\ref{vbs1}) and (\ref{vbs2}). Set%
\[%
\begin{array}
[c]{l}%
I_{1}(t)=y^{\varepsilon}(t)-\bar{y}^{\varepsilon}(t),\text{ }I_{2}%
(t)=z^{\varepsilon}(t)-\bar{z}^{\varepsilon}(t)-\langle p(t),\delta
\sigma(t)\rangle I_{E_{\varepsilon}}(t),\\
I_{3}(t)=\langle p(t),\delta b(t)\rangle+\langle q(t),\delta\sigma
(t)\rangle+\frac{1}{2}\langle P(t)\delta\sigma(t),\delta\sigma(t)\rangle,\\
I_{4}(t)=f(t,x^{\varepsilon}(t),y^{\varepsilon}(t),z^{\varepsilon
}(t),u^{\varepsilon}(t))-f(t,\bar{x}(t)+x_{1}(t)+x_{2}(t),\bar{y}%
^{\varepsilon}(t)+I_{1}(t),\bar{z}^{\varepsilon}(t)+I_{2}(t),\bar{u}(t)),\\
I_{5}(t)=\delta\sigma_{x}^{T}(t)p(t)+\frac{1}{2}P(t)\delta\sigma(t)+\frac
{1}{2}P^{T}(t)\delta\sigma(t),\\
I_{6}(t)=[I_{n\times n},p(t),\sigma_{x}^{T}(t)p(t)+q(t)].
\end{array}
\]
Then we can get%
\begin{equation}%
\begin{array}
[c]{cl}%
\bar{y}^{\varepsilon}(t)-\bar{y}(t)= & o(\varepsilon)+\int_{t}^{T}%
\{I_{3}(s)I_{E_{\varepsilon}}(s)+I_{4}(s)+\tilde{f}_{y}(s)(\bar{y}%
^{\varepsilon}(s)-\bar{y}(s))+\tilde{f}_{z}(s)(\bar{z}^{\varepsilon}%
(s)-\bar{z}(s))\\
& +[(x_{1}(s)+x_{2}(s))^{T},I_{1}(s),I_{2}(s)]\tilde{D}^{2}f(s)[(x_{1}%
(s)+x_{2}(s))^{T},I_{1}(s),I_{2}(s)]^{T}\\
& +f_{z}(s)\langle I_{5}(s),x_{1}(s)\rangle I_{E_{\varepsilon}}(s)-\frac{1}%
{2}\langle I_{6}(s)D^{2}f(s)I_{6}^{T}(s)x_{1}(s),x_{1}(s)\rangle\}ds\\
& -\int_{t}^{T}(\bar{z}^{\varepsilon}(s)-\bar{z}(s))dW(s),
\end{array}
\label{vbs11}%
\end{equation}
where%
\[%
\begin{array}
[c]{l}%
f(t,\bar{x}(t)+x_{1}(t)+x_{2}(t),\bar{y}^{\varepsilon}(t)+I_{1}(t),\bar
{z}^{\varepsilon}(t)+I_{2}(t),\bar{u}(t))\\
-f(t,\bar{x}(t)+x_{1}(t)+x_{2}(t),\bar{y}(t)+I_{1}(t),\bar{z}(t)+I_{2}%
(t),\bar{u}(t))\\
=\tilde{f}_{y}(s)(\bar{y}^{\varepsilon}(s)-\bar{y}(s))+\tilde{f}_{z}%
(s)(\bar{z}^{\varepsilon}(s)-\bar{z}(s)),
\end{array}
\]%
\[%
\begin{array}
[c]{rl}%
\tilde{f}_{y}(s)= & \int_{0}^{1}f_{y}(s,\bar{x}(s)+x_{1}(s)+x_{2}(s),\bar
{y}(s)+I_{1}(s)+\mu(\bar{y}^{\varepsilon}(s)-\bar{y}(s)),\\
& \bar{z}(s)+I_{2}(s)+\mu(\bar{z}^{\varepsilon}(s)-\bar{z}(s)),\bar{u}%
(s))d\mu,\\
\tilde{f}_{z}(s)= & \int_{0}^{1}f_{z}(s,\bar{x}(s)+x_{1}(s)+x_{2}(s),\bar
{y}(s)+I_{1}(s)+\mu(\bar{y}^{\varepsilon}(s)-\bar{y}(s)),\\
& \bar{z}(s)+I_{2}(s)+\mu(\bar{z}^{\varepsilon}(s)-\bar{z}(s)),\bar{u}%
(s))d\mu,\\
\tilde{D}^{2}f(s)= & \int_{0}^{1}\int_{0}^{1}\lambda D^{2}f(s,\bar
{x}(s)+\lambda\mu(x_{1}(s)+x_{2}(s)),\bar{y}(s)+\lambda\mu I_{1}(s),\\
& \bar{z}(s)+\lambda\mu I_{2}(s),\bar{u}(s))d\lambda d\mu.
\end{array}
\]
Thus by Theorem \ref{th1}, equation (\ref{vbs11}) and standard estimates of
BSDEs, we can easily obtain (\ref{vbs1}) and (\ref{vbs2}). It is obviously for
equation (\ref{vbs3}). Now we prove (\ref{vbs4}). Set%
\[
\tilde{x}^{\varepsilon}(t)=x^{\varepsilon}(t)-\bar{x}(t)-x_{1}(t)-x_{2}%
(t),\text{ }\tilde{y}^{\varepsilon}(t)=\bar{y}^{\varepsilon}(t)-\bar
{y}(t)-\hat{y}(t),\text{ }\tilde{z}^{\varepsilon}(t)=\bar{z}^{\varepsilon
}(t)-\bar{z}(t)-\hat{z}(t).
\]
By equations (\ref{vbs11}) and (\ref{vbsde1}), we can get%
\begin{equation}%
\begin{array}
[c]{cl}%
\tilde{y}^{\varepsilon}(t)= & o(\varepsilon)+\int_{t}^{T}\{\tilde{f}%
_{y}(s)\tilde{y}^{\varepsilon}(s)+\tilde{f}_{z}(s)\tilde{z}^{\varepsilon
}(s)+(\tilde{f}_{y}(s)-f_{y}(s))\hat{y}(s)+(\tilde{f}_{z}(s)-f_{z}(s))\hat
{z}(s)\\
& +I_{4}(s)-[f(s,\bar{x}(s),\bar{y}(s),\bar{z}(s)+\langle p(s),\delta
\sigma(s)\rangle,u)-f(s)]I_{E_{\varepsilon}}(s)\\
& +[(x_{1}(s)+x_{2}(s))^{T},I_{1}(s),I_{2}(s)]\tilde{D}^{2}f(s)[(x_{1}%
(s)+x_{2}(s))^{T},I_{1}(s),I_{2}(s)]^{T}\\
& +f_{z}(s)\langle I_{5}(s),x_{1}(s)\rangle I_{E_{\varepsilon}}(s)-\frac{1}%
{2}\langle I_{6}(s)D^{2}f(s)I_{6}^{T}(s)x_{1}(s),x_{1}(s)\rangle\}ds\\
& -\int_{t}^{T}\tilde{z}^{\varepsilon}(s)dW(s).
\end{array}
\label{vbs12}%
\end{equation}
By Theorem \ref{th1}, it is easy to check that we only need to show that%
\[%
\begin{array}
[c]{l}%
E[(\int_{0}^{T}|(\tilde{f}_{y}(s)-f_{y}(s))\hat{y}(s)+(\tilde{f}_{z}%
(s)-f_{z}(s))\hat{z}(s)|ds)^{2}]=o(\varepsilon^{2}),\\
E[(\int_{0}^{T}|\langle I_{6}(s)(\tilde{D}^{2}f(s)-\frac{1}{2}D^{2}%
f(s))I_{6}^{T}(s)x_{1}(s),x_{1}(s)\rangle|ds)^{2}]=o(\varepsilon^{2}),\\
E[(\int_{0}^{T}|I_{4}(s)-[f(s,\bar{x}(s),\bar{y}(s),\bar{z}(s)+\langle
p(s),\delta\sigma(s)\rangle,u)-f(s)]I_{E_{\varepsilon}}(s)|ds)^{2}%
]=o(\varepsilon^{2}).
\end{array}
\]
Note that
\[
|\tilde{f}_{y}(s)-f_{y}(s)|+|\tilde{f}_{z}(s)-f_{z}(s)|\leq C(|x_{1}%
(s)+x_{2}(s)|+|I_{1}(s)|+|I_{2}(s)|+|\bar{y}^{\varepsilon}(s)-\bar
{y}(s)|+|\bar{z}^{\varepsilon}(s)-\bar{z}(s)|),
\]
and%
\begin{equation}%
\begin{array}
[c]{cl}%
E[(\int_{0}^{T}|q(s)x_{1}(s)\hat{z}(s)|ds)^{2}] & \leq E[\sup_{s\in
\lbrack0,T]}|x_{1}(s)|^{2}(\int_{0}^{T}|q(s)|^{2}ds)(\int_{0}^{T}|\hat
{z}(s)|^{2}ds)]\\
& \leq(E[(\int_{0}^{T}|\hat{z}(s)|^{2}ds)^{2}])^{1/2}(E[\sup_{s\in\lbrack
0,T]}|x_{1}(s)|^{8}])^{1/4}(E[(\int_{0}^{T}|q(s)|^{2}ds)^{4}])^{1/4}\\
& =o(\varepsilon^{2}),
\end{array}
\label{vbs13}%
\end{equation}
we can get $E[(\int_{0}^{T}|(\tilde{f}_{y}(s)-f_{y}(s))\hat{y}(s)+(\tilde
{f}_{z}(s)-f_{z}(s))\hat{z}(s)|ds)^{2}]=o(\varepsilon^{2})$. Since $D^{2}f$ is
bounded, we can get that for each $\beta\geq2$,%
\[
E[(\int_{0}^{T}|(\tilde{D}^{2}f(s)-\frac{1}{2}D^{2}f(s))||q(s)|^{2}ds)^{\beta
}]\rightarrow0\text{ as }\varepsilon\rightarrow0\text{.}%
\]
Thus we can easily deduce $E[(\int_{0}^{T}|\langle I_{6}(s)(\tilde{D}%
^{2}f(s)-\frac{1}{2}D^{2}f(s))I_{6}^{T}(s)x_{1}(s),x_{1}(s)\rangle
|ds)^{2}]=o(\varepsilon^{2})$. It is easy to verify that%
\[%
\begin{array}
[c]{l}%
|I_{4}(s)-[f(s,\bar{x}(s),\bar{y}(s),\bar{z}(s)+\langle p(s),\delta
\sigma(s)\rangle,u)-f(s)]I_{E_{\varepsilon}}(s)|\\
\leq C\{|\tilde{x}^{\varepsilon}(s)|+[|x_{1}(s)+x_{2}(s)|+|\bar{y}%
^{\varepsilon}(t)-\bar{y}(t)|+|\bar{z}^{\varepsilon}(s)-\bar{z}(s)|+|I_{1}%
(s)|+|I_{2}(s)|]I_{E_{\varepsilon}}(s)\}.
\end{array}
\]
Since%
\begin{align*}
E[(\int_{0}^{T}|q(s)x_{1}(s)|I_{E_{\varepsilon}}(s)ds)^{2}] &  \leq
E[\sup_{s\in\lbrack0,T]}|x_{1}(s)|^{2}\int_{E_{\varepsilon}}|q(s)|^{2}%
ds]\varepsilon\\
&  \leq(E[\sup_{s\in\lbrack0,T]}|x_{1}(s)|^{4}])^{1/2}(E[(\int_{E_{\varepsilon
}}|q(s)|^{2}ds)^{2}])^{1/2}\varepsilon\\
&  =o(\varepsilon^{2}),
\end{align*}
we can obtain $E[(\int_{0}^{T}|I_{4}(s)-[f(s,\bar{x}(s),\bar{y}(s),\bar
{z}(s)+\langle p(s),\delta\sigma(s)\rangle,u)-f(s)]I_{E_{\varepsilon}%
}(s)|ds)^{2}]=o(\varepsilon^{2})$. The proof is complete.
\end{proof}

Thus we obtain the following variational equation for BSDE (\ref{state2}):%
\begin{equation}%
\begin{array}
[c]{rl}%
y^{\varepsilon}(t)= & \bar{y}(t)+\langle p(t),x_{1}(t)+x_{2}(t)\rangle
+\frac{1}{2}\langle P(t)x_{1}(t),x_{1}(t)\rangle+\hat{y}(t)+o(\varepsilon),\\
z^{\varepsilon}(t)= & \bar{z}(t)+\langle p(t),\delta\sigma(t)\rangle
I_{E_{\varepsilon}}(t)+\langle\sigma_{x}^{T}(t)p(t)+q(t),x_{1}(t)+x_{2}%
(t)\rangle\\
& +\langle\delta\sigma_{x}^{T}(t)p(t)+\frac{1}{2}P(t)\delta\sigma(t)+\frac
{1}{2}P^{T}(t)\delta\sigma(t),x_{1}(t)\rangle I_{E_{\varepsilon}}(t)\\
& +\frac{1}{2}\langle\lbrack\sigma_{xx}^{T}(t)p(t)+P(t)\sigma_{x}%
(t)+\sigma_{x}^{T}(t)P(t)+Q(t)]x_{1}(t),x_{1}(t)\rangle+\hat{z}%
(t)+o(\varepsilon).
\end{array}
\label{veq123}%
\end{equation}

\begin{remark}
\label{newrem123} We can also give the variational equations for BSDE
(\ref{state2}) as in \cite{P90}. Set
\begin{equation}
y_{1}(t)=\langle p(t),x_{1}(t)\rangle,\text{ }z_{1}(t)=\langle p(t),\delta
\sigma(t)\rangle I_{E_{\varepsilon}}(t)+\langle\sigma_{x}^{T}%
(t)p(t)+q(t),x_{1}(t)\rangle,\label{neweq12}%
\end{equation}
it is easy to check that $(y_{1},z_{1})$ satisfies the following BSDE:%
\begin{equation}
\left\{
\begin{array}
[c]{rl}%
-dy_{1}(t)= & \{\langle f_{x}(t),x_{1}(t)\rangle+f_{y}(t)y_{1}(t)+f_{z}%
(t)z_{1}(t)\\
& -[f_{z}(t)\langle p(t),\delta\sigma(t)\rangle+\langle q(t),\delta
\sigma(t)\rangle]I_{E_{\varepsilon}}(t)\}dt-z_{1}(t)dW(t),\\
y_{1}(T)= & \langle\phi_{x}(\bar{x}(T)),x_{1}(T)\rangle.
\end{array}
\right.  \label{neweq13}%
\end{equation}
Set%
\begin{equation}%
\begin{array}
[c]{rl}%
y_{2}(t)= & \langle p(t),x_{2}(t)\rangle+\frac{1}{2}\langle P(t)x_{1}%
(t),x_{1}(t)\rangle+\hat{y}(t),\\
z_{2}(t)= & \langle\sigma_{x}^{T}(t)p(t)+q(t),x_{2}(t)\rangle+\langle
\delta\sigma_{x}^{T}(t)p(t)+\frac{1}{2}P(t)\delta\sigma(t)+\frac{1}{2}%
P^{T}(t)\delta\sigma(t),x_{1}(t)\rangle I_{E_{\varepsilon}}(t)\\
& +\frac{1}{2}\langle\lbrack\sigma_{xx}^{T}(t)p(t)+P(t)\sigma_{x}%
(t)+\sigma_{x}^{T}(t)P(t)+Q(t)]x_{1}(t),x_{1}(t)\rangle+\hat{z}(t),
\end{array}
\label{neweq14}%
\end{equation}
it is easy to verify that%
\begin{equation}
\left\{
\begin{array}
[c]{rl}%
-dy_{2}(t)= & \{\langle f_{x}(t),x_{2}(t)\rangle+f_{y}(t)y_{2}(t)+f_{z}%
(t)z_{2}(t)\\
& +\frac{1}{2}[(x_{1}(t))^{T},y_{1}(t),z_{1}(t)]D^{2}f(t)[(x_{1}(t))^{T}%
,y_{1}(t),z_{1}(t)]^{T}\\
& +[\langle q(t),\delta\sigma(t)\rangle-\frac{1}{2}f_{zz}(t)(\langle
p(t),\delta\sigma(t)\rangle)^{2}+f(t,\bar{x}(t),\bar{y}(t),\bar{z}(t)+\langle
p(t),\delta\sigma(t)\rangle,u)\\
& -f(t,\bar{x}(t),\bar{y}(t),\bar{z}(t),\bar{u}(t))]I_{E_{\varepsilon}%
}(t)+\langle L(t),x_{1}(t)\rangle I_{E_{\varepsilon}}(t)\}dt-z_{2}(t)dW(t),\\
y_{2}(T)= & \langle\phi_{x}(\bar{x}(T)),x_{2}(T)\rangle+\frac{1}{2}\langle
\phi_{xx}(\bar{x}(T))x_{1}(T),x_{1}(T)\rangle,
\end{array}
\right.  \label{neweq15}%
\end{equation}
where $\langle L(t),x_{1}(t)\rangle I_{E_{\varepsilon}}(t)=o(\varepsilon)$, so
we do not give the explicit formula for $L(t)$. Here we use
\[
\lbrack f(t,\bar{x}(t),\bar{y}(t),\bar{z}(t)+\langle p(t),\delta
\sigma(t)\rangle,u)-f(t,\bar{x}(t),\bar{y}(t),\bar{z}(t),\bar{u}%
(t))]I_{E_{\varepsilon}}(t)
\]
to completely deal with the term $\langle p(t),\delta\sigma(t)\rangle
I_{E_{\varepsilon}}(t)$ in the variation of $z$, so the terms $f_{z}(t)\langle
p(t),\delta\sigma(t)\rangle$ and $\frac{1}{2}f_{zz}(t)(\langle p(t),\delta
\sigma(t)\rangle)^{2}$ are repeated in $f_{z}(t)z_{1}(t)$ and $\frac{1}%
{2}f_{zz}(t)(z_{1}(t))^{2}$. Noting equations (\ref{neweq12}) and
(\ref{neweq14}), then the adjoint equations for $(z_{1}(t))^{2}$ and other
terms are essentially for $x_{1}(t)$, $x_{2}(t)$ and $x_{1}(t)(x_{1}(t))^{T}$,
which is solved in \cite{P90}. In order to further explain the difference of
expansions for SDE and BSDE, we consider the following equations:%
\begin{equation}
\left\{
\begin{array}
[c]{rl}%
-d\tilde{y}_{1}(t)= & \{\langle f_{x}(t),x_{1}(t)\rangle+f_{y}(t)\tilde{y}%
_{1}(t)+f_{z}(t)\tilde{z}_{1}(t)\}dt-\tilde{z}_{1}(t)dW(t),\\
\tilde{y}_{1}(T)= & \langle\phi_{x}(\bar{x}(T)),x_{1}(T)\rangle,
\end{array}
\right.  \label{neweq16}%
\end{equation}%
\begin{equation}
\left\{
\begin{array}
[c]{rl}%
-d\tilde{y}_{2}(t)= & \{\langle f_{x}(t),x_{2}(t)\rangle+f_{y}(t)\tilde{y}%
_{2}(t)+f_{z}(t)\tilde{z}_{2}(t)\\
& +\frac{1}{2}[(x_{1}(t))^{T},\tilde{y}_{1}(t),\tilde{z}_{1}(t)]D^{2}%
f(t)[(x_{1}(t))^{T},\tilde{y}_{1}(t),\tilde{z}_{1}(t)]^{T}\\
& +[f(t,\bar{x}(t),\bar{y}(t),\bar{z}(t)+\langle p(t),\delta\sigma
(t)\rangle,u)-f(t,\bar{x}(t),\bar{y}(t),\bar{z}(t),\bar{u}(t))\\
& -f_{z}(t)\langle p(t),\delta\sigma(t)\rangle-\frac{1}{2}f_{zz}(t)(\langle
p(t),\delta\sigma(t)\rangle)^{2}]I_{E_{\varepsilon}}(t)\}dt-\tilde{z}%
_{2}(t)dW(t),\\
\tilde{y}_{2}(T)= & \langle\phi_{x}(\bar{x}(T)),x_{2}(T)\rangle+\frac{1}%
{2}\langle\phi_{xx}(\bar{x}(T))x_{1}(T),x_{1}(T)\rangle.
\end{array}
\right.  \label{neweq17}%
\end{equation}
By the standard estimates of BSDEs, it is easy to show that
\[
E[\sup_{t\in\lbrack0,T]}|y_{1}(t)+y_{2}(t)-\tilde{y}_{1}(t)-\tilde{y}%
_{2}(t)|^{2}+\int_{0}^{T}|z_{1}(t)+z_{2}(t)-\tilde{z}_{1}(t)-\tilde{z}%
_{2}(t)|^{2}dt]=o(\varepsilon^{2}).
\]
Thus by equation (\ref{veq123}), we can get
\[%
\begin{array}
[c]{l}%
y^{\varepsilon}(t)=\bar{y}(t)+\tilde{y}_{1}(t)+\tilde{y}_{2}(t)+o(\varepsilon
),\\
z^{\varepsilon}(t)=\bar{z}(t)+\tilde{z}_{1}(t)+\tilde{z}_{2}(t)+o(\varepsilon
).
\end{array}
\]
The main difference is equation (\ref{neweq17}) which is due to the term
$\langle p(t),\delta\sigma(t)\rangle I_{E_{\varepsilon}}(t)$ in the variation
of $z$. If $f$ is independent of $z$, the variational equations for $(y,z)$
are the same as in \cite{P90}, which is pointed in \cite{P98}.
\end{remark}

Now we consider the maximum principle. From equation (\ref{veq123}), we get%
\[
J(u^{\varepsilon}(\cdot))-J(\bar{u}(\cdot))=y^{\varepsilon}(0)-\bar{y}%
(0)=\hat{y}(0)+o(\varepsilon).
\]
Define the following adjoint equation for BSDE (\ref{vbsde1}):%
\[
\left\{
\begin{array}
[c]{l}%
d\gamma(t)=f_{y}(t)\gamma(t)dt+f_{z}(t)\gamma(t)dW(t),\\
\gamma(0)=1.
\end{array}
\right.
\]
Applying It\^{o}'s formula to $\gamma(t)\hat{y}(t)$, we can obtain%
\begin{equation}%
\begin{array}
[c]{rl}%
\hat{y}(0)= & E[\int_{0}^{T}\gamma(s)[\langle p(s),\delta b(s)\rangle+\langle
q(s),\delta\sigma(s)\rangle+\frac{1}{2}\langle P(s)\delta\sigma(s),\delta
\sigma(s)\rangle\\
& +f(s,\bar{x}(s),\bar{y}(s),\bar{z}(s)+\langle p(s),\delta\sigma
(s)\rangle,u)-f(s,\bar{x}(s),\bar{y}(s),\bar{z}(s),\bar{u}%
(s))]I_{E_{\varepsilon}}(s)ds].
\end{array}
\label{mpeq1}%
\end{equation}
Note that $\gamma(s)>0$, then we define the following function:%
\begin{equation}%
\begin{array}
[c]{rl}%
\mathcal{H}(t,x,y,z,u,p,q,P)= & \langle p,b(t,x,u)\rangle+\langle
q,\sigma(t,x,u)\rangle\\
& +\frac{1}{2}\langle P(\sigma(t,x,u)-\sigma(t,\bar{x},\bar{u})),\sigma
(t,x,u)-\sigma(t,\bar{x},\bar{u})\rangle\\
& +f(t,x,y,z+\langle p,\sigma(t,x,u)-\sigma(t,\bar{x},\bar{u})\rangle,u),
\end{array}
\label{ham123}%
\end{equation}
where $(p,q,P)$ is defined in equations (\ref{adjoint1}) and (\ref{adjoint2}).
Thus we obtain the following maximum principle.

\begin{theorem}
\label{newth4}Suppose (A1) and (A2) hold. Let $\bar{u}(\cdot)$ be an optimal
control and $(\bar{x}(\cdot),\bar{y}(\cdot),\bar{z}(\cdot))$ be the
corresponding solution. Then%
\begin{equation}
\mathcal{H}(t,\bar{x}(t),\bar{y}(t),\bar{z}(t),u,p(t),q(t),P(t))\geq
\mathcal{H}(t,\bar{x}(t),\bar{y}(t),\bar{z}(t),\bar{u}%
(t),p(t),q(t),P(t)),\forall u\in U,\text{a.e., a.s.,} \label{newmp123}%
\end{equation}
where $\mathcal{H}(\cdot)$ is defined in (\ref{ham123}).
\end{theorem}

If the control domain $U$ is convex, we can get the following corollary which
is obtained by Peng in \cite{P93}.

\begin{corollary}
Let the assumptions as in Theorem \ref{newth4}. If $U$ is convex and $b$,
$\sigma$, $f$ are continuously differentiable with respect to $u$, then%
\[
\langle b_{u}^{T}(t)p(t)+\sigma_{u}^{T}(t)q(t)+f_{z}(t)\sigma_{u}%
^{T}(t)p(t)+f_{u}(t),u-\bar{u}(t)\rangle\geq0,\text{ }\forall u\in
U,a.e.,a.s..
\]

\end{corollary}

Now we give an example to compare our result with the result in \cite{W, Y1}.

\begin{example}
\label{newexa1}Suppose $n=d=k=1$. $U$ is a given subset in $\mathbb{R}$.
Consider the following control system:%
\[
dx(t)=u(t)dW(t),\text{ }x(0)=0,
\]%
\[
y(t)=x(T)+\int_{t}^{T}f(z(s))ds-\int_{t}^{T}z(s)dW(s).
\]
In this case, our maximum principle is
\begin{equation}
f(\bar{z}(t)+u-\bar{u}(t))-f(\bar{z}(t))\geq0,\text{ }\forall u\in
U,a.e.,a.s.. \label{exa1}%
\end{equation}
Note that%
\[
y(t)-\int_{0}^{t}u(s)dW(s)=\int_{t}^{T}f(z(s)-u(s)+u(s))ds-\int_{t}%
^{T}(z(s)-u(s))dW(s),
\]
then by comparison theorem of BSDE, it is easy to check that inequality
(\ref{exa1}) is a sufficient condition. For the case $U=\{0,1\}$, $f(0)=0$,
$f^{\prime}(0)<0$, $f(1)>0$, $f(-1)<0$, it is easy to verify that $(\bar
{x},\bar{y},\bar{z},\bar{u})=(0,0,0,0)$ satisfies (\ref{exa1}), thus $\bar
{u}=0$ is an optimal control. But $f_{z}(\bar{z}(t))(1-\bar{u}(t))<0$, which
implies that $\bar{u}=0$ is not an optimal control for the case $U=[0,1]$. The
maximum principle in \cite{Y1} is $f_{z}(\bar{z}(t))(u-\bar{u}(t))\geq0$,
$\forall u\in U$, $a.e.$, $a.s.$, which only cover the case $U$ is convex. The
maximum principle in \cite{W} contains two unknown parameters.
\end{example}

\begin{remark}
In \cite{W, Y1}, the authors consider the control system which consists of SDE
(\ref{state1}) and the following state equation:%
\begin{equation}
y(t)=y_{0}-\int_{0}^{t}f(s,x(s),y(s),v(s),u(s))ds+\int_{0}^{t}v(s)dW(s),
\label{state3}%
\end{equation}
where the set of all admissible controls%
\[
\mathcal{\tilde{U}}[0,T]=\{(u,y_{0},v)\in\mathcal{U}[0,T]\times\mathbb{R}%
\times M^{2}(0,T):y(T)=\phi(x(T))\}.
\]
The optimal control problem is to minimize $J(u(\cdot),y_{0},v(\cdot))=y_{0}$
over $\mathcal{\tilde{U}}[0,T]$. Obviously, this problem is equivalent to
Peng's problem. Thus our maximum principle also completely solves this control problem.
\end{remark}

\subsection{Multi-dimensional case}

In this sebsection, we extend Peng's problem to multi-dimensional case, i.e.,
the functions in BSDE (\ref{state2}) are $m$-dimensional, $\phi:\mathbb{R}%
^{n}\rightarrow\mathbb{R}^{m}$, $f:[0,T]\times\mathbb{R}^{n}\times
\mathbb{R}^{m}\times\mathbb{R}^{m\times d}\times\mathbb{R}^{k}\rightarrow
\mathbb{R}^{m}$. The cost functional is defined by%
\begin{equation}
J(u(\cdot))=h(y(0)),\label{newcost3}%
\end{equation}
where $h:\mathbb{R}^{m}\rightarrow\mathbb{R}$. For deriving the variational
equation for BSDE (\ref{state2}), we use the following notation.%
\begin{equation}%
\begin{array}
[c]{l}%
W(t)=(W^{1}(t),\ldots,W^{d}(t))^{T},\text{ }\phi(x)=(\phi^{1}(x),\ldots
,\phi^{m}(x))^{T},\\
\sigma(t,x,u)=(\sigma^{ij}(t,x,u)),\text{ }i=1,\ldots,n,j=1,\ldots,d,\\
\sigma^{j}(t,x,u)=(\sigma^{1j}(t,x,u),\ldots,\sigma^{nj}(t,x,u))^{T}%
,j=1,\ldots,d,\\
f(t,x,y,z,u)=(f^{1}(t,x,y,z,u),\ldots,f^{m}(t,x,y,z,u))^{T},\\
y(t)=(y^{1}(t),\ldots,y^{m}(t))^{T},z(t)=(z^{ij}(t)),i\leq m,j\leq d,\\
z^{j}(t)=(z^{1j}(t),\ldots,z^{mj}(t))^{T},j=1,\ldots,d.
\end{array}
\label{nota3}%
\end{equation}
We introduce the following adjoint equations: for $i=1,\ldots,m$,%
\begin{equation}
\left\{
\begin{array}
[c]{rl}%
-dp_{i}(t)= & F_{i}(t)dt-\sum_{j=1}^{d}q_{i}^{j}(t)dW^{j}(t),\\
p_{i}(T)= & \phi_{x}^{i}(\bar{x}(T)),
\end{array}
\right.  \label{adjoint6}%
\end{equation}%
\begin{equation}
\left\{
\begin{array}
[c]{rl}%
-dP_{i}(t)= & G_{i}(t)dt-\sum_{j=1}^{d}Q_{i}^{j}(t)dW^{j}(t),\\
P_{i}(T)= & \phi_{xx}^{i}(\bar{x}(T)),
\end{array}
\right.  \label{adjoint7}%
\end{equation}
where $F_{i}(t)$ and $G_{i}(t)$ is given after the following notations:%
\begin{equation}%
\begin{array}
[c]{l}%
p(t)=[p_{1}(t),\ldots,p_{m}(t)]_{n\times m}\text{, }q^{j}(t)=[q_{1}%
^{j}(t),\ldots,q_{m}^{j}(t)]_{n\times m},\\
p_{l}(t)=(p_{l}^{1}(t),\ldots,p_{l}^{n}(t))^{T},\text{ }q_{l}^{j}%
(t)=(q_{l}^{1j}(t),\ldots,q_{l}^{nj}(t))^{T},\text{ }\\
b_{xx}^{T}(t)p_{l}(t)=\sum_{i=1}^{n}p_{l}^{i}(t)(b_{xx}^{i}(t))^{T},\text{
}(\sigma_{xx}^{j}(t))^{T}p_{l}(t)=\sum_{i=1}^{n}p_{l}^{i}(t)(\sigma_{xx}%
^{ij}(t))^{T},\\
(\sigma_{xx}^{j}(t))^{T}q_{l}^{j}(t)=\sum_{i=1}^{n}q_{l}^{ij}(t)(\sigma
_{xx}^{ij}(t))^{T},l=1,\ldots,m,\text{ }j=1,\ldots,d.
\end{array}
\label{nota4}%
\end{equation}%
\begin{equation}%
\begin{array}
[c]{rl}%
F_{i}(t)= & b_{x}^{T}(t)p_{i}(t)+f_{x}^{i}(t)+\sum_{l=1}^{m}f_{y^{l}}%
^{i}(t)p_{l}(t)+\sum_{j=1}^{d}(\sigma_{x}^{j}(t))^{T}q_{i}^{j}(t)\\
& +\sum_{j=1}^{d}\sum_{l=1}^{m}f_{z^{lj}}^{i}(t)[(\sigma_{x}^{j}(t))^{T}%
p_{l}(t)+q_{l}^{j}(t)],\\
G_{i}(t)= & P_{i}(t)b_{x}(t)+(b_{x}(t))^{T}P_{i}(t)+\sum_{l=1}^{m}f_{y^{l}%
}^{i}(t)P_{l}(t)+\sum_{j=1}^{d}[Q_{i}^{j}(t)\sigma_{x}^{j}(t)+(\sigma_{x}%
^{j}(t))^{T}Q_{i}^{j}(t)\\
& +(\sigma_{xx}^{j}(t))^{T}q_{i}^{j}(t)+(\sigma_{x}^{j}(t))^{T}P_{i}%
(t)\sigma_{x}^{j}(t)]+\sum_{j=1}^{d}\sum_{l=1}^{m}[f_{z^{lj}}^{i}%
(t)P_{l}(t)\sigma_{x}^{j}(t)+f_{z^{lj}}^{i}(t)(\sigma_{x}^{j}(t))^{T}%
P_{l}(t)\\
& +f_{z^{lj}}^{i}(t)Q_{l}^{j}(t)+f_{z^{lj}}^{i}(t)(\sigma_{xx}^{j}%
(t))^{T}p_{l}(t)]+b_{xx}^{T}(t)p_{i}(t)+[I_{n\times n},p(t),(\sigma_{x}%
^{1}(t))^{T}p(t)+q^{1}(t),\ldots,\\
& (\sigma_{x}^{d}(t))^{T}p(t)+q^{d}(t)]D^{2}f^{i}(t)[I_{n\times n}%
,p(t),(\sigma_{x}^{1}(t))^{T}p(t)+q^{1}(t),\ldots,(\sigma_{x}^{d}%
(t))^{T}p(t)+q^{d}(t)]^{T},
\end{array}
\label{fun1}%
\end{equation}
where $D^{2}f^{i}$ is the Hessian matrix of $f^{i}$ with respect to
$(x,y,z^{1},\ldots,z^{d})$. Let $\hat{y}(t)=(\hat{y}^{1}(t),\ldots,\hat{y}%
^{m}(t))^{T}$, $\hat{z}(t)=(\hat{z}^{ij}(t))$ be the solution of the following
BSDE:%
\begin{equation}%
\begin{array}
[c]{rl}%
\hat{y}(t)= & \int_{t}^{T}[f_{y}(s)\hat{y}(s)+\sum_{j=1}^{d}f_{z^{j}}%
(s)\hat{z}^{j}(s)+\{p^{T}(s)\delta b(s)+\sum_{j=1}^{d}[(q^{j}(s))^{T}%
\delta\sigma^{j}(s)+\frac{1}{2}P^{T}(s)\delta\sigma^{j}(s)\delta\sigma
^{j}(s)]\\
& +f(s,\bar{x}(s),\bar{y}(s),\bar{z}(s)+p^{T}(s)\delta\sigma(s),u)-f(s,\bar
{x}(s),\bar{y}(s),\bar{z}(s),\bar{u}(s))\}I_{E_{\varepsilon}}(s)]ds\\
& -\sum_{j=1}^{d}\int_{t}^{T}\hat{z}^{j}(s)dW^{j}(s),
\end{array}
\label{newvb12}%
\end{equation}
where%
\[
P(t)=[P_{1}(t),\ldots,P_{m}(t)],\text{ }P^{T}(s)\delta\sigma^{j}%
(s)\delta\sigma^{j}(s)=(\langle P_{1}(s)\delta\sigma^{j}(s),\delta\sigma
^{j}(s)\rangle,\ldots,\langle P_{m}(s)\delta\sigma^{j}(s),\delta\sigma
^{j}(s)\rangle)^{T}.
\]
Similar to the analysis in Theorem \ref{th3}, we can get the following
variational principle:%
\begin{equation}%
\begin{array}
[c]{rl}%
y^{i;\varepsilon}(t)= & \bar{y}^{i}(t)+\langle p_{i}(t),x_{1}(t)+x_{2}%
(t)\rangle+\frac{1}{2}\langle P_{i}(t)x_{1}(t),x_{1}(t)\rangle+\hat{y}%
^{i}(t)+o(\varepsilon),\\
z^{ij;\varepsilon}(t)= & \bar{z}^{ij}(t)+\langle p_{i}(t),\delta\sigma
^{j}(t)\rangle I_{E_{\varepsilon}}(t)+\langle(\sigma_{x}^{j}(t))^{T}%
p_{i}(t)+q_{i}^{j}(t),x_{1}(t)+x_{2}(t)\rangle\\
& +\langle(\delta\sigma_{x}^{j}(t))^{T}p_{i}(t)+\frac{1}{2}P_{i}%
(t)\delta\sigma^{j}(t)+\frac{1}{2}P_{i}^{T}(t)\delta\sigma^{j}(t),x_{1}%
(t)\rangle I_{E_{\varepsilon}}(t)\\
& +\frac{1}{2}\langle\lbrack(\sigma_{xx}^{j}(t))^{T}p_{i}(t)+P_{i}%
(t)\sigma_{x}^{j}(t)+(\sigma_{x}^{j}(t))^{T}P_{i}(t)+Q_{i}^{j}(t)]x_{1}%
(t),x_{1}(t)\rangle\\
& +\hat{z}^{ij}(t)+o(\varepsilon),\text{ }i=1,\ldots,m,\text{ }j=1,\ldots,d.
\end{array}
\label{multivari}%
\end{equation}
Let $h\in C^{1}(\mathbb{R}^{m})$. Then we get%
\[
J(u^{\varepsilon}(\cdot))-J(\bar{u}(\cdot))=\langle h_{y}(\bar{y}(0)),\hat
{y}(0)\rangle+o(\varepsilon).
\]
We introduce the following adjoint equation for BSDE (\ref{newvb12}).%
\begin{equation}
\left\{
\begin{array}
[c]{l}%
d\gamma(t)=f_{y}^{T}(t)\gamma(t)dt+\sum_{j=1}^{d}f_{z^{j}}^{T}(t)\gamma
(t)dW^{j}(t),\\
\gamma(0)=h_{y}(\bar{y}(0)).
\end{array}
\right.  \label{adjoint8}%
\end{equation}
Applying It\^{o}'s formula to $\langle\gamma(t),\hat{y}(t)\rangle$, we can get
the following maximum principle.

\begin{theorem}
\label{th5}Suppose (A1) and (A2) hold. Let $\bar{u}(\cdot)$ be an optimal
control and $(\bar{x}(\cdot),\bar{y}(\cdot),\bar{z}(\cdot))$ be the
corresponding solution. The cost function is defined in (\ref{newcost3}) and
$h\in C^{1}(\mathbb{R}^{m})$. Then%
\begin{equation}%
\begin{array}
[c]{l}%
\langle\gamma(t),p^{T}(t)\delta b(t)+\sum_{j=1}^{d}[(q^{j}(t))^{T}\delta
\sigma^{j}(t)+\frac{1}{2}P^{T}(t)\delta\sigma^{j}(t)\delta\sigma^{j}(t)]\\
+f(t,\bar{x}(t),\bar{y}(t),\bar{z}(t)+p^{T}(t)\delta\sigma(t),u)-f(s,\bar
{x}(s),\bar{y}(s),\bar{z}(s),\bar{u}(s))\rangle\\
\geq0,\text{ \ }\forall u\in U,\text{a.e., a.s.,}%
\end{array}
\label{multimp}%
\end{equation}
where $p$, $q^{j}$, $P$, $\gamma$ are given in equations (\ref{adjoint6}),
(\ref{adjoint7}), (\ref{fun1}) and (\ref{adjoint8}).
\end{theorem}

\section{Problem with state constraint}

For the simplicity of presentation, suppose $d=m=1$, the multi-dimensional
case can be treated with the same method.

We consider the control system: SDE (\ref{state1}) and BSDE (\ref{state2}).
The cost function $J(u(\cdot))$ is defined in (\ref{cost2}). In addition, we
consider the following state constraint:%
\begin{equation}
E[\varphi(x(T),y(0))]=0, \label{statecons}%
\end{equation}
where $\varphi:\mathbb{R}^{n}\times\mathbb{R}\rightarrow\mathbb{R}$. We need
the following assumption:

\begin{description}
\item[(A3)] $\varphi$ is twice continuously differentiable with respect to
$(x,y)$; $D^{2}\varphi$ is bounded; $D\varphi$ is bounded by $C(1+|x|+|y|)$.
\end{description}

Define all admissible controls as follows:%
\[
\mathcal{U}_{ad}[0,T]=\{u(\cdot)\in\mathcal{U}[0,T]:E[\varphi
(x(T),y(0))]=0\}.
\]
The control problem is to minimize $J(u(\cdot))$ over $\mathcal{U}_{ad}[0,T]$.

Let $\bar{u}(\cdot)\in\mathcal{U}_{ad}[0,T]$ be an optimal control and
$(\bar{x}(\cdot),\bar{y}(\cdot),\bar{z}(\cdot))$ be the corresponding solution
of equations (\ref{state1}) and (\ref{state2}). Similarly, we define
$(x(\cdot),y(\cdot),z(\cdot),u(\cdot))$ for any $u(\cdot)\in\mathcal{U}[0,T]$.
For any $\rho>0$, define the following cost functional on $\mathcal{U}[0,T]$:%
\begin{equation}
J_{\rho}(u(\cdot))=\{[(y(0)-\bar{y}(0))+\rho]^{2}+|E[\varphi(x(T),y(0))]|^{2}%
\}^{1/2}.\label{costcons}%
\end{equation}
It is easy to check that%
\[
\left\{
\begin{array}
[c]{l}%
J_{\rho}(u(\cdot))>0,\text{ }\forall u(\cdot)\in\mathcal{U}[0,T],\\
J_{\rho}(\bar{u}(\cdot))=\rho\leq\inf_{u\in\mathcal{U}[0,T]}J_{\rho}%
(u(\cdot))+\rho.
\end{array}
\right.
\]
In order to use well-known Ekeland's variational principle, we define the
following metric on $\mathcal{U}[0,T]$:%
\[
d(u(\cdot),v(\cdot))=E[\int_{0}^{T}I_{\{u\not =v\}}(t,\omega)dt].
\]
Suppose that $(\mathcal{U}[0,T],d)$ is a complete space and $J_{\rho}(\cdot)$
is continuous, otherwise we can use the technique in \cite{TL, W} and the
result is the same. Thus, by Ekeland's variational principle, there exists a
$u_{\rho}(\cdot)\in\mathcal{U}[0,T]$ such that%
\begin{equation}%
\begin{array}
[c]{l}%
J_{\rho}(u_{\rho}(\cdot))\leq\rho,\text{ }d(u_{\rho}(\cdot),\bar{u}%
(\cdot))\leq\sqrt{\rho},\\
J_{\rho}(u(\cdot))-J_{\rho}(u_{\rho}(\cdot))+\sqrt{\rho}d(u_{\rho}%
(\cdot),u(\cdot))\geq0,\text{ }\forall u(\cdot)\in\mathcal{U}[0,T].
\end{array}
\label{cons1}%
\end{equation}
For any $\varepsilon>0$, let $E_{\varepsilon}\subset\lbrack0,T]$ with
$|E_{\varepsilon}|=\varepsilon$, define%
\[
u_{\rho}^{\varepsilon}(t)=u_{\rho}(t)I_{E_{\varepsilon}^{c}}%
(t)+uI_{E_{\varepsilon}}(t),\text{ }\forall u\in U.
\]
It is easy to check that $d(u_{\rho}(\cdot),u_{\rho}^{\varepsilon}(\cdot
))\leq\varepsilon$. Let $(x_{\rho}(\cdot),y_{\rho}(\cdot),z_{\rho}(\cdot))$ be
the solution corresponding to $u_{\rho}(\cdot)$. Similarly for $(x_{\rho
}^{\varepsilon}(\cdot),y_{\rho}^{\varepsilon}(\cdot),z_{\rho}^{\varepsilon
}(\cdot),u_{\rho}^{\varepsilon}(\cdot))$. Thus by (\ref{cons1}), we can get%
\begin{equation}%
\begin{array}
[c]{cl}%
0 & \leq J_{\rho}(u_{\rho}^{\varepsilon}(\cdot))-J_{\rho}(u_{\rho}%
(\cdot))+\sqrt{\rho}\varepsilon\\
& \leq\lambda_{\rho}[y_{\rho}^{\varepsilon}(0)-y_{\rho}(0)]+\mu_{\rho
}\{E[\varphi(x_{\rho}^{\varepsilon}(T),y_{\rho}^{\varepsilon}(0))]-E[\varphi
(x_{\rho}(T),y_{\rho}(0))]\}+\sqrt{\rho}\varepsilon+o(\varepsilon),
\end{array}
\label{cons2}%
\end{equation}
where%
\[
\lambda_{\rho}=J_{\rho}(u_{\rho}(\cdot))^{-1}[(y_{\rho}(0)-\bar{y}%
(0))+\rho],\text{ }\mu_{\rho}=J_{\rho}(u_{\rho}(\cdot))^{-1}E[\varphi(x_{\rho
}(T),y_{\rho}(0))].
\]
The same analysis as in Theorem \ref{th3}, let $(p^{\rho}(\cdot),q^{\rho
}(\cdot))$ and $(P^{\rho}(\cdot),Q^{\rho}(\cdot))$ be respectively the
solutions of equations (\ref{adjoint1}) and (\ref{adjoint2}) with $(\bar
{x}(\cdot),\bar{y}(\cdot),\bar{z}(\cdot),\bar{u}(\cdot))$ replaced by
$(x_{\rho}(\cdot),y_{\rho}(\cdot),z_{\rho}(\cdot),u_{\rho}(\cdot))$, and let
all the coefficients be added by a superscript $\rho$. Then%
\begin{equation}
y_{\rho}^{\varepsilon}(0)-y_{\rho}(0)=\hat{y}_{\rho}(0)+o(\varepsilon
),\label{cons21}%
\end{equation}
where%
\begin{equation}%
\begin{array}
[c]{cl}%
\hat{y}_{\rho}(t)= & \int_{t}^{T}\{f_{y}^{\rho}(s)\hat{y}_{\rho}%
(s)+f_{z}^{\rho}(s)\hat{z}_{\rho}(s)+[\langle p^{\rho}(s),\delta b^{\rho
}(s)\rangle+\langle q^{\rho}(s),\delta\sigma^{\rho}(s)\rangle+\frac{1}%
{2}\langle P^{\rho}(s)\delta\sigma^{\rho}(s),\delta\sigma^{\rho}(s)\rangle\\
& +f(s,x_{\rho}(s),y_{\rho}(s),z_{\rho}(s)+\langle p^{\rho}(s),\delta
\sigma^{\rho}(s)\rangle,u)-f(s,x_{\rho}(s),y_{\rho}(s),z_{\rho}(s),u_{\rho
}(s))]I_{E_{\varepsilon}}(s)\}ds\\
& -\int_{t}^{T}\hat{z}_{\rho}(s)dW(s).
\end{array}
\label{cons3}%
\end{equation}
Similarly, let
\[
\left\{
\begin{array}
[c]{rl}%
-dp_{0}^{\rho}(t)= & [(b_{x}^{\rho}(t))^{T}p_{0}^{\rho}(t)+(\sigma_{x}^{\rho
}(t))^{T}q_{0}^{\rho}(t)]dt-q_{0}^{\rho}(t)dW(t),\\
p_{0}(T)= & \mu_{\rho}\varphi_{x}(x_{\rho}(T),y_{\rho}(0)),
\end{array}
\right.
\]%
\[
\left\{
\begin{array}
[c]{rl}%
-dP_{0}^{\rho}(t)= & [(b_{x}^{\rho}(t))^{T}P_{0}^{\rho}(t)+P_{0}^{\rho
}(t)b_{x}^{\rho}(t)+(\sigma_{x}^{\rho}(t))^{T}P_{0}^{\rho}(t)\sigma_{x}^{\rho
}(t)+(\sigma_{x}^{\rho}(t))^{T}Q_{0}^{\rho}(t)+Q_{0}^{\rho}(t)\sigma_{x}%
^{\rho}(t)\\
& +(b_{xx}^{\rho}(t))^{T}p_{0}^{\rho}(t)+(\sigma_{xx}^{\rho}(t))^{T}%
q_{0}^{\rho}(t)]dt-Q_{0}^{\rho}(t)dW(t),\\
P_{0}(T)= & \mu_{\rho}\varphi_{xx}(x_{\rho}(T),y_{\rho}(0)).
\end{array}
\right.
\]
Then%
\begin{equation}%
\begin{array}
[c]{l}%
\mu_{\rho}\{E[\varphi(x_{\rho}^{\varepsilon}(T),y_{\rho}^{\varepsilon
}(0))]-E[\varphi(x_{\rho}(T),y_{\rho}(0))]\}\\
=E[\int_{0}^{T}\{\langle p_{0}^{\rho}(s),\delta b^{\rho}(s)\rangle+\langle
q_{0}^{\rho}(s),\delta\sigma^{\rho}(s)\rangle+\frac{1}{2}\langle P_{0}^{\rho
}(s)\delta\sigma^{\rho}(s),\delta\sigma^{\rho}(s)\rangle\}I_{E_{\varepsilon}%
}(s)ds]\\
\text{ \ }+\mu_{\rho}E[\varphi_{y}(x_{\rho}(T),y_{\rho}(0))]\hat{y}_{\rho
}(0)+o(\varepsilon).
\end{array}
\label{cons4}%
\end{equation}
Define the following adjoint equation for BSDE (\ref{cons3}):%
\[
\left\{
\begin{array}
[c]{l}%
d\gamma^{\rho}(t)=f_{y}^{\rho}(t)\gamma^{\rho}(t)dt+f_{z}^{\rho}%
(t)\gamma^{\rho}(t)dW(t),\\
\gamma^{\rho}(0)=\lambda_{\rho}+\mu_{\rho}E[\varphi_{y}(x_{\rho}(T),y_{\rho
}(0))].
\end{array}
\right.
\]
Then we can get%
\begin{equation}%
\begin{array}
[c]{l}%
\{\lambda_{\rho}+\mu_{\rho}E[\varphi_{y}(x_{\rho}(T),y_{\rho}(0))]\}\hat
{y}_{\rho}(0)\\
=E[\int_{0}^{T}\gamma^{\rho}(s)\{\langle p^{\rho}(s),\delta b^{\rho}%
(s)\rangle+\langle q^{\rho}(s),\delta\sigma^{\rho}(s)\rangle+\frac{1}%
{2}\langle P^{\rho}(s)\delta\sigma^{\rho}(s),\delta\sigma^{\rho}(s)\rangle\\
\text{ \ }+f(s,x_{\rho}(s),y_{\rho}(s),z_{\rho}(s)+\langle p^{\rho}%
(s),\delta\sigma^{\rho}(s)\rangle,u)-f(s,x_{\rho}(s),y_{\rho}(s),z_{\rho
}(s),u_{\rho}(s))\}I_{E_{\varepsilon}}(s)ds].
\end{array}
\label{cons5}%
\end{equation}
Define the following function:
\begin{equation}%
\begin{array}
[c]{l}%
\mathcal{H}(t,x,y,z,u,x^{\prime},u^{\prime},p_{0},q_{0},P_{0},p,q,P,\gamma)\\
=\langle p_{0}+\gamma p,b(t,x,u)\rangle+\langle q_{0}+\gamma q,\sigma
(t,x,u)\rangle\\
\text{ \ }+\frac{1}{2}\langle(P_{0}+\gamma P)(\sigma(t,x,u)-\sigma
(t,x^{\prime},u^{\prime})),\sigma(t,x,u)-\sigma(t,x^{\prime},u^{\prime
})\rangle\\
\text{ \ }+\gamma(t)f(t,x,y,z+\langle p,\sigma(t,x,u)-\sigma(t,x^{\prime
},u^{\prime})\rangle,u).
\end{array}
\label{cons8}%
\end{equation}
It follows from (\ref{cons2}), (\ref{cons21}), (\ref{cons4}) and (\ref{cons5})
that%
\begin{align*}
0 &  \leq E[\int_{0}^{T}\{\mathcal{H}(t,x_{\rho}(t),y_{\rho}(t),z_{\rho
}(t),u,x_{\rho}(t),u_{\rho}(t),p_{0}^{\rho}(t),q_{0}^{\rho}(t),P_{0}^{\rho
}(t),p^{\rho}(t),q^{\rho}(t),P^{\rho}(t),\gamma^{\rho}(t))\\
&  \text{ \ \ }-\mathcal{H}(t,x_{\rho}(t),y_{\rho}(t),z_{\rho}(t),u_{\rho
}(t),x_{\rho}(t),u_{\rho}(t),p_{0}^{\rho}(t),q_{0}^{\rho}(t),P_{0}^{\rho
}(t),p^{\rho}(t),q^{\rho}(t),P^{\rho}(t),\gamma^{\rho}(t))\}I_{E_{\varepsilon
}}(t)dt]\\
&  \text{ \ \ }+\sqrt{\rho}\varepsilon+o(\varepsilon).
\end{align*}
Thus we obtain%
\[%
\begin{array}
[c]{l}%
\mathcal{H}(t,x_{\rho}(t),y_{\rho}(t),z_{\rho}(t),u,x_{\rho}(t),u_{\rho
}(t),p_{0}^{\rho}(t),q_{0}^{\rho}(t),P_{0}^{\rho}(t),p^{\rho}(t),q^{\rho
}(t),P^{\rho}(t),\gamma^{\rho}(t))\\
\geq\mathcal{H}(t,x_{\rho}(t),y_{\rho}(t),z_{\rho}(t),u_{\rho}(t),x_{\rho
}(t),u_{\rho}(t),p_{0}^{\rho}(t),q_{0}^{\rho}(t),P_{0}^{\rho}(t),p^{\rho
}(t),q^{\rho}(t),P^{\rho}(t),\gamma^{\rho}(t))\\
\text{ \ }-\sqrt{\rho},\text{ }\forall u\in U,\text{ a.e., a.s..}%
\end{array}
\]
Obviously, $|\lambda_{\rho}|^{2}+|\mu_{\rho}|^{2}=1$. Thus there exists a
subsequence of $(\lambda_{\rho},\mu_{\rho})$ which converges to $(\lambda
,\mu)$ with $|\lambda|^{2}+|\mu|^{2}=1$ as $\rho\rightarrow0$. Note that
$d(u_{\rho}(\cdot),\bar{u}(\cdot))\leq\sqrt{\rho}$, then we can get for
further subsequence%
\[%
\begin{array}
[c]{l}%
(x_{\rho}(\cdot),y_{\rho}(\cdot),z_{\rho}(\cdot),u_{\rho}(\cdot),p_{0}^{\rho
}(\cdot),q_{0}^{\rho}(\cdot),P_{0}^{\rho}(\cdot),p^{\rho}(\cdot),q^{\rho
}(\cdot),P^{\rho}(\cdot),\gamma^{\rho}(\cdot))\rightarrow\\
(\bar{x}(\cdot),\bar{y}(\cdot),\bar{z}(\cdot),\bar{u}(\cdot),p_{0}%
(\cdot),q_{0}(\cdot),P_{0}(\cdot),p(\cdot),q(\cdot),P(\cdot),\gamma
(\cdot)),\text{ a.e., a.s.,}%
\end{array}
\]
where $(p(\cdot),q(\cdot))$ and $(P(\cdot),Q(\cdot))$ are respectively the
solutions of equations (\ref{adjoint1}) and (\ref{adjoint2}),%
\begin{equation}
\left\{
\begin{array}
[c]{rl}%
-dp_{0}(t)= & [(b_{x}(t))^{T}p_{0}(t)+(\sigma_{x}(t))^{T}q_{0}(t)]dt-q_{0}%
(t)dW(t),\\
p_{0}(T)= & \mu\varphi_{x}(\bar{x}(T),\bar{y}(0)),
\end{array}
\right.  \label{cons9}%
\end{equation}%
\begin{equation}
\left\{
\begin{array}
[c]{rl}%
-dP_{0}(t)= & [(b_{x}(t))^{T}P_{0}(t)+P_{0}(t)b_{x}(t)+(\sigma_{x}%
(t))^{T}P_{0}(t)\sigma_{x}(t)+(\sigma_{x}(t))^{T}Q_{0}(t)+Q_{0}(t)\sigma
_{x}(t)\\
& +(b_{xx}(t))^{T}p_{0}(t)+(\sigma_{xx}(t))^{T}q_{0}(t)]dt-Q_{0}(t)dW(t),\\
P_{0}(T)= & \mu\varphi_{xx}(\bar{x}(T),\bar{y}(0)),
\end{array}
\right.  \label{cons10}%
\end{equation}%
\begin{equation}
\left\{
\begin{array}
[c]{l}%
d\gamma(t)=f_{y}(t)\gamma(t)dt+f_{z}(t)\gamma(t)dW(t),\\
\gamma(0)=\lambda+\mu E[\varphi_{y}(\bar{x}(T),\bar{y}(0))].
\end{array}
\right.  \label{cons11}%
\end{equation}
Thus we get the following theorem.

\begin{theorem}
Suppose (A1), (A2) and (A3) hold. Let $\bar{u}(\cdot)$ be an optimal control
with state constraint (\ref{statecons}) and $(\bar{x}(\cdot),\bar{y}%
(\cdot),\bar{z}(\cdot))$ be the corresponding solution. Then there exist two
contants $\lambda$, $\mu$ with $|\lambda|^{2}+|\mu|^{2}=1$ such that
\[%
\begin{array}
[c]{l}%
\mathcal{H}(t,\bar{x}(t),\bar{y}(t),\bar{z}(t),u,\bar{x}(t),\bar{u}%
(t),p_{0}(t),q_{0}(t),P_{0}(t),p(t),q(t),P(t),\gamma(t))\\
\geq\mathcal{H}(t,\bar{x}(t),\bar{y}(t),\bar{z}(t),\bar{u},\bar{x}(t),\bar
{u}(t),p_{0}(t),q_{0}(t),P_{0}(t),p(t),q(t),P(t),\gamma(t)),\\
\text{
\ \ \ \ \ \ \ \ \ \ \ \ \ \ \ \ \ \ \ \ \ \ \ \ \ \ \ \ \ \ \ \ \ \ \ \ \ \ \ \ }%
\forall u\in U,\text{ a.e., a.s.,}%
\end{array}
\]
where $\mathcal{H}(\cdot)$, $(p(\cdot),q(\cdot))$, $(P(\cdot),Q(\cdot))$,
$(p_{0}(\cdot),q_{0}(\cdot))$, $(P_{0}(\cdot),Q_{0}(\cdot))$ and $\gamma
(\cdot)$ are defined in (\ref{cons8}), (\ref{adjoint1}), (\ref{adjoint2}),
(\ref{cons9}), (\ref{cons10}) and (\ref{cons11}).
\end{theorem}

\bigskip

\textbf{Acknowledgments}

I would like to thank Professor S. Peng for many helpful discussions and
valuable comments. I also would like to thank Professor S. Ji for many helpful discussions.


\begin{thebibliography}{99}                                                                                               %


\bibitem {CE}Z. Chen and L. Epstein, Ambiguity, risk, and asset returns in
continuous time, Econometrica, 70 (2002), pp. 1403-1443.

\bibitem {DZ}M. Dokuchaev and {\normalsize X. Y. Zhou, Stochastic controls
with terminal contingent conditions, J. Math. Anal. Appl., 238 (1999), pp.
143-165.}

\bibitem {DE}D. Duffie and L. Epstein, \emph{Stochastic differential utility},
Econometrica, 60 (1992), pp. 353--394.

\bibitem {EPQ}N. El Karoui, S. Peng and M. C. Quenez, \emph{Backward
stochastic differential equations in finance}, Math. Finance, 7 (1997), pp. 1-71.

\bibitem {EPQ1}N. El Karoui, S. Peng and M. C. Quenez, A dynamic maximum
priciple for the optimization of recursive utilities under constraints, Ann.
Appl. Probab., 11 (2001), pp. 664-693.

\bibitem {HP}Y. Hu and S. Peng, Solution of forward-backward stochastic
differential equations, Probab. Theory Related Fields, 103 (1995), pp. 273-283.

\bibitem {JZ}S. Ji and {\normalsize X. Y. Zhou, A maximum principle for
stochastic optimal control with terminal state constrains, and its
applications, Comm. Inf. Syst., 6 (2006), pp. 321-337.}

\bibitem {MY}J. Ma and J. Yong, Forward-Backward Stochastic Differential
Equations and Their Applications, Springer-Verlag, Berlin, 1999.

\bibitem {KZ}M. Kohlmann and {\normalsize X. Y.} Zhou, Relationship between
backward stochastic differential equations and stochastic controls: a
linear-quadratic approach, SIAM J. Control Optim., 38 (2000), pp. 1392-1407.

\bibitem {LZ}A. Lim and {\normalsize X. Y.} Zhou, Linear-quadratic control of
backward stochastic differential equations, SIAM J. Control Optim., 40 (2001),
pp. 450-474.

\bibitem {PP90}E. Pardoux and S. Peng, \emph{Adapted Solutions of Backward
Stochastic Equations,} Systerm and Control Letters, 14 (1990), pp. 55-61.

\bibitem {P90}S. Peng, A general stochastic maximum principle for optimal
control problems, SIAM J. Control Optim., 28 (1990), pp. 966-979.

\bibitem {P93}S. Peng, Backward stochastic differential equations and
applications to optimal control, Appl. Math. Optim., 27 (1993), pp. 125-144.

\bibitem {P97}S. Peng, Backward SDE and related $g$-expectation, in Backward
Stochastic Differential Equations (Paris, 1995-1996), Pitman Res. Notes Math.
Ser. 364, Longman, Harlow, 1997, pp. 141-159.

\bibitem {P98}S. Peng, Open problems on backward stochastic differential
equations, In S. Chen, X. Li, J. Yong and {\normalsize X. Y.} Zhou (Eds),
Control of distributed parameter and stocastic systems, (1998), pp. 265-273.

\bibitem {PW}S. Peng and Z. Wu, Fully coupled forward-backward stochastic
differential equations and applications to optimal control, SIAM J. Control
Optim., 37 (1999), pp. 825-843.

\bibitem {SW}J. Shi and Z. Wu, The maximum principle for fully coupled
forward-backward stochastic control system, Acta Automat. Sinica, 32 (2006),
pp. 161-169.

\bibitem {TL}S. Tang and X. Li, Necessary conditions for optimal control of
stochastic systems with random jumps, SIAM J. Control Optim., 32 (1994), pp. 1447-1475.

\bibitem {W1}Z. Wu, Maximum principle for optimal control problem of fully
coupled forward-backward stochastic systems, Systems Sci. Math. Sci., 11
(1998), pp. 249-259.

\bibitem {W}Z. Wu, A general maximum principle for optimal control of
forward-backward stochastic systems, Automatica, 49 (2013), 1473-1480.

\bibitem {X}W. Xu, Stochastic maximum principle for optimal control problem of
forward and backward system, J. Austral. Math. Soc. Ser. B, 37 (1995), pp. 172-185.

\bibitem {Y0}J. Yong, Stochastic optimal control and forward-backward
stochastic differential equations, Comput. Appl. Math., 21 (2002), pp. 369-403.

\bibitem {Y}J. Yong, Forward-backward stochastic differential equations with
mixed initial and terminal conditions, Trans. Amer. Math. Soc., 362 (2010),
pp. 1047-1096.

\bibitem {Y1}J. Yong, Optimality variational principle for controlled
forward-backward stochastic differential equations with mixed initial-terminal
conditions, SIAM J. Control Optim., 48 (2010), pp. 4119-4156.

\bibitem {YZ}{\normalsize J. Yong and X. Y. Zhou, \emph{Stochastic controls:
Hamiltonian systems and HJB equations}, Springer-Verlag, New York, 1999. }
\end{thebibliography}
\end{document}